\documentclass[reqno,11pt]{amsart}
\usepackage{amssymb, amsmath,amssymb, amscd, amsbsy,latexsym}
\usepackage{amsthm, amsfonts,mathrsfs}
\usepackage[T1]{fontenc}
\usepackage{times}
\usepackage{epsfig}
\usepackage{color}
\input amssym.def
\input amssym.tex

\def\inte#1{
\displaystyle\mathop{#1\kern0pt}^\circ }

\def\virgp{\raise 2pt\hbox{,}}
\def\cdotpv{\raise 2pt\hbox{;}}

\def\C{\mathop{\mathbb C\kern 0pt}\nolimits}
\def\DD{\mathop{\mathbb D\kern 0pt}\nolimits}
\def\EE{\mathop{{\mathbb E \kern 0pt}}\nolimits}
\def\K{\mathop{\mathbb K\kern 0pt}\nolimits}
\def\N{\mathop{\mathbb N\kern 0pt}\nolimits}
\def\Q{\mathop{\mathbb Q\kern 0pt}\nolimits}
\def\R{\mathop{\mathbb R\kern 0pt}\nolimits}
\def\SS{\mathop{\mathbb S\kern 0pt}\nolimits}
\def\ZZ{\mathop{\mathbb Z\kern 0pt}\nolimits}
\def\TT{\mathop{\mathbb T\kern 0pt}\nolimits}
\def\P{\mathop{\mathbb P\kern 0pt}\nolimits}

\newcommand{\beq}{\begin{equation}}
\newcommand{\eeq}{\end{equation}}
\newcommand{\ben}{\begin{eqnarray}}
\newcommand{\een}{\end{eqnarray}}
\newcommand{\beno}{\begin{eqnarray*}}
\newcommand{\eeno}{\end{eqnarray*}}

\renewcommand{\theequation}{\thesection.\arabic{equation}}
\newtheorem{remark}{Remark}[section]
\newtheorem{lemma}{Lemma}[section]

\newtheorem{theorem}{Theorem}[section]

\newdimen\eqjot \eqjot = 1\jot
\def\openupeq{\openup \the\eqjot}

\def\pofbox#1 #2$#3${\setbox0=\hbox{$#3$}\ht0=0pt\dp0=0pt\wd0=0pt\hskip-#1pt\raise#2pt\box0\hskip#1pt}

\begin{document}

\title[Orbital stability of peakons for a higher-order $\mu$-CH equation]
{Orbital stability of periodic peakons for a new higher-order $\mu$-Camassa-Holm equation}
\author{Gezi Chong}
\address{Gezi Chong\newline
Center for Nonlinear Studies and School of Mathematics\\
Northwest University\\
Xi'an 710127\\
P. R. China\\
} \email{geziChongxajd@163.com}
\author{Ying Fu}
\address{Ying Fu\newline
Center for Nonlinear Studies and School of Mathematics\\
Northwest University\\
Xi'an 710127\\
P. R. China\\
} \email{fuying@nwu.edu.cn}
\date{\today}
\maketitle

\begin{abstract}
Consideration here is a higher-order $\mu$-Camassa-Holm equation, which is a higher-order extension
of the $\mu$-Camassa-Holm equation and retains some properties of the $\mu$-Camassa-Holm equation
and the modified $\mu$-Camassa-Holm equation. By utilizing the inequalities with the maximum and
minimum of the solution related to the first three conservation laws, we establish that the periodic peakons
of this equation are orbitally stable under small perturbations in the energy space.
\end{abstract}


\vskip 0.2cm


\renewcommand{\theequation}{\thesection.\arabic{equation}}
\setcounter{equation}{0}

\section{Introduction}
In this paper, we consider the following new higher-order $\mu$-Camassa-Holm (CH) equation \cite{wlfq}
\begin{equation}\label{mma}
\begin{split}
u_{t}&+\sum\limits_{i=1}^{3}a_{i}u^{i}u_{x}
+(\mu-\partial_{x}^{2})^{-1}\partial_{x}\Bigg(\sum\limits_{i=1}^{3}a_{i}u^{i-1}\left(\left(i+1\right)\mu(u)u+\frac{1}{2}iu_{x}^{2}\right)\\
&-\frac{23}{24}\sum\limits_{i=1}^{2}(i+1)a_{i+1}\mu^{2}(u)u^{i}\Bigg)=0,\\
\end{split}
\end{equation}
where $u(t,x)$ is a time-dependent but spatially periodic function over the unit circle
$\mathbb{S}^{1}=\mathbb{R}/\mathbb{Z}$, $\mu(u)=\int_{\mathbb{S}^{1}}u(t,x)dx$ denotes its mean,
and $a_{1}$, $a_{2}$, $a_{3}$ are constants.
This model equation was introduced as a higher-order extension of the $\mu$-CH equation
by Wang et al. in \cite{wlfq}.

The well-known Camassa-Holm (CH) equation \cite{ch} (derived earlier in \cite{ff})
\begin{equation}\label{mme}
m_{t}+2mu_{x}+um_{x}=0,\quad\quad m=u-u_{xx},
\end{equation}
which was proposed as a model for the unidirectional propagation of the shallow water waves over a flat
bottom, and $u(t,x)$ stands for the fluid velocity at time $t$ in the spatial direction $x$ \cite{ch,cl,ion}.
This equation has many important properties such as complete integrability and
infinitely many conservation laws \cite{ch,ff}, wave breaking phenomena \cite{ce,lo}, geometric formulations
\cite{cq,kou,mis} and the existence of peaked solitons and multi-peakons \cite{achm,ch,cht}. Furthermore, the results of the analysis
of the CH equation \eqref{mme} could be also found in \cite{bc,bou,dan,lp,wl} and references therein.
A short-wave limit of the CH equation is the integrable Hunter-Saxton (HS) equation \cite{hs}
\begin{equation*}
u_{xt}+uu_{xx}+\frac{1}{2}u_{x}^{2}=0.
\end{equation*}
A midway equation between the CH equation and HS equation is the $\mu$-CH equation, that is,
Eq. \eqref{mma} with $a_{2}=a_{3}=0$ and $a_{1}=1$ reduces to the $\mu$-CH
equation \cite{klm}
\begin{equation}\label{mmc}
m_{t}+2mu_{x}+um_{x}=0,\quad\quad m=\mu(u)-u_{xx}.
\end{equation}
This equation was proposed as an integrable equation arising in the study
of the diffeomorphism group of the circle by Khesin, Lenells and Misiolek \cite{klm}.
It describes the propagation of weakly nonlinear orientation
waves in a massive nematic liquid crystals with external magnetic field and self-interaction.
Eq. \eqref{mmc} can also be viewed as a geodesic equation with respect to a right invariant
Riemannian metric induced by the $\mu$ inner product \cite{klm}
\begin{equation*}
\langle f,g\rangle_{\mu}=\mu(f)\mu(g)+\int_{\mathbb{S}^{1}}f'(x)g'(x)dx.
\end{equation*}
More interestingly, Eq. \eqref{mmc} is formally integrable, since it admits a Lax formulation
and has the bi-Hamiltonian structure. Its integrability, blow-up, global existence,
well-posedness, single peakon and multi-peakons were discussed in \cite{flq,klm,lmt}.

Another celebrated integrable equation admitting peakons is the Degasperis-Procesi (DP) equation
\begin{equation}\label{mmd}
m_{t}+um_{x}+3u_{x}m=0,\quad\quad m=u-u_{xx}.
\end{equation}
It is regarded as a model for nonlinear shallow water dynamics and its asymptotic accuracy is the same as for the CH shallow
water equation \cite{cl}. Degasperis, Kholm and Khon \cite{dkk} proved that the DP equation \eqref{mmd} is completely integrable
by constructing a Lax pair. In \cite{dkk}, it was also shown that Eq. \eqref{mmd} has a bi-Hamiltonian structure with an infinite sequence
of conserved quantities. Its blow-up mechanism and global existence have been presented in \cite{ely,ely-1,ly}. Besides, the local well-posedness
of the initial-value problem of the DP equation in Besov spaces and Sobolev spaces was established in \cite{gl,yin}.
A short wave limit of the DP equation is the integrable Burgers equation. A midway equation between the DP equation and
Burgers equation is the $\mu$-DP equation of the form \eqref{mmd} with $m=(\mu-\partial_{x}^{2})u$, which was derived by Lenells,
Misiolek and Ti$\breve{g}$lay \cite{lmt}. Its integrability, well-posedness, existence of peakons and blow-up were also verified in \cite{flq,lmt}.

It is worth noting that all nonlinear terms in the $\mu$-CH and $\mu$-DP equations are quadratic.
Recently, some $\mu$-CH-type equations with cubic nonlinearity have been introduced.
For example, the nonlocal counterpart of the modified Camassa-Holm (mCH) equation, that is,
the modified $\mu$-CH equation with cubic nonlinearity is as follows:
\begin{equation}\label{mmg}
m_{t}+((2\mu(u)u-u_{x}^{2})m)_{x}=0,\quad\quad m=\mu(u)-u_{xx},
\end{equation}
which was introduced in \cite{qfl}. Furthermore, it was demonstrated that Eq. \eqref{mmg} arises from a non-stretching planar curve flow
in Euclidean geometry \cite{qfl}. Moreover, the existence of peaked traveling waves, wave-breaking and
local well-posedness were established in \cite{lqz,qfl,qfl-1}.

In addition, various higher-order extensions of the $\mu$-CH-type equations have also been proposed. In \cite{wlfq},
a new $\mu$-version of the higher-order CH equation \eqref{mma} was derived by Wang, Luo, Fu and Qu.
It preserves some typical properties of the $\mu$-CH and modified $\mu$-CH equations, and can be obtained
from the existence of peaked solitons, $H_{\mu}$-conserved density and a higher-order conserved density.
Compared with the $\mu$-CH equation, the model admits much stronger local nonlinearity.
It was shown in \cite{wlfq} that Eq. \eqref{mma} possesses the peaked traveling wave solutions. Meanwhile, the local well-posedness,
blow-up criterion and wave breaking mechanism have also been discussed in \cite{wlfq}.

One remarkable property of the CH equation \eqref{mme}
is the presence of peaked soliton solutions, called peakons. It is presented by
\begin{equation*}
u(t,x)=\varphi_{c}(x-ct)=ae^{-|x-ct|},
\end{equation*}
where the amplitude $a$ is given by the constant $c$ for the CH equation \cite{ch}. Its corresponding periodic peakon takes the form
\begin{equation*}
u(t,x)=\varphi_{c}(x-ct)=a\frac{\cosh\left(x-ct-[x-ct]-\frac{1}{2}\right)}{\cosh\left(\frac{1}{2}\right)},
\end{equation*}
where the notation $[\zeta]$ denotes the greatest integer part of $\zeta$, and the amplitude $a$ is also given by
$c$ for the CH equation.

Moreover, it was noticed that the $\mu$-CH equation \cite{klm,lmt}
and the modified $\mu$-CH equation \cite{qfl} also admit periodic peakons of the form:
\begin{equation}\label{mmh}
u(x,t)=\varphi_{c}(x-ct)=a\varphi(x-ct),
\end{equation}
where
\begin{equation}\label{lblb}
\varphi(x)=\frac{1}{2}\left(x^{2}+\frac{23}{12}\right),\quad\quad x\in\left[-\frac{1}{2},\frac{1}{2}\right],
\end{equation}
and $\varphi$ is extended periodically to the real line, and the amplitude $a$ takes values $12c/13$
and $2\sqrt{3c}/5$, respectively, for the $\mu$-CH equation and modified $\mu$-CH equation.
Analogous to the above equations, the higher-order $\mu$-CH equation \eqref{mma} \cite{wlfq} possesses the peaked periodic traveling
wave solution, which takes the form of \eqref{mmh} with $a$ satisfying
\begin{equation}\label{mmi}
\sum\limits_{i=1}^{3}12^{3-i}13^{i}a^{i}a_{i}-12^{3}c=0.
\end{equation}

Motivated by the works in \cite{cs,len}, the aim of this paper is to verify the stability of periodic peaked solitons of
a $\mu$-version of the higher-order CH equation \eqref{mma}. In particular, the explicit structure of the peakons ensures
that the differences of the maximum height and $H^{1}$ energies determines the $H^{1}$ difference between any $H^{1}$
function $u$ and a bounded peakon $\varphi_{c}$, that is,
\begin{equation*}
\|u-\varphi_{c}\|_{H^{1}}^{2}\leq|H_{1}[u]-H_{1}[\varphi_{c}]|+|M_{u}-M_{\varphi_{c}}|.
\end{equation*}
For a solution $u$ to the higher-order $\mu$-CH equation, by use of the conservation law $H_{1}[u]$ the problem
of the orbital stability further reduces to the one of stability of the control
of $|M_{u}-M_{\varphi_{c}}|$, which can be achieved by a construction of a Lyapunov function through utilizing
the other two conservation laws.

If waves such as the peakons are to be observable in nature, they need to be stable under small perturbations.
Therefore, the stability of the peakons is of great interest. In an innovative paper \cite{cs},
Constantin and Strauss proved that the single peakons of the CH equation \eqref{mme} are orbitally stable by utilizing the conservation
laws and the feature of the peakons. This approach was further developed in \cite{ll,qll}.
Constantin and Molinet \cite{cm} proposed a variational approach to demonstrate the orbital stability of the peakons.
Applying the methods in \cite{cs,mmt}, it was shown in \cite{dm} that the trains of peaked solitons of the CH equation \eqref{mme}
are orbitally stable. Meanwhile, in \cite{len}, Lenells studied the stability of the periodic peaked solitons for the CH equation.
The approach in \cite{len} was recently extended to investigate the orbital stability of the periodic peakons for
classical integrable equations such as the $\mu$-CH equation \cite{cll}, the modified CH equation \cite{qll},
the modified $\mu$-CH equation \cite{lqz}, the generalized $\mu$-CH equation \cite{qzll} and the generalized modified CH equation \cite{moon-1}.

In this paper, we shall verify the following stability of periodic peaked solitons of
a $\mu$-version of the higher-order CH equation \eqref{mma}.

\begin{theorem}\label{t3.1}
Assume that $a>0$ and $a_{i}$, $i=1,2,3$ satisfy one of the following five conditions

$(i)$ $a_{1}>0$, $a_{2}>0$, $a_{3}>0$.

$(ii)$ $a_{2}<0$, $a_{2}^{2}<3a_{1}a_{3}$, $a_{3}>0$.

$(iii)$ $a_{1}>0$, $a_{2}>-\frac{6a_{1}}{13a}$, $a_{3}=0$.

$(iv)$ $a_{1}=0$, $a_{2}>0$, $a_{3}=0$.

$(v)$ $a_{1}<0$, $a_{2}>-\frac{a_{1}}{2a}$, $a_{3}=0$,

\noindent with the wave speed $c$ satisfies \eqref{mmi}. For every $\varepsilon>0$, there exists $\delta>0$ such that if $u\in C([0,T),H^{s}(\mathbb{S}^{1}))\cap C^{1}([0,T),H^{s-1}(\mathbb{S}^{1}))$,
$s>\frac{3}{2}$ is a solution to Eq. \eqref{mma} with
\begin{equation*}
\|u(\cdot,0)-\varphi_{c}\|_{H^{1}(\mathbb{S}^{1})}<\delta,
\end{equation*}
then
\begin{equation*}
\left\|u(\cdot,t)-\varphi_{c}\left(\cdot-\xi(t)+\frac{1}{2}\right)\right\|_{H^{1}(\mathbb{S}^{1})}<\varepsilon,
\quad\quad\quad \text{for}\quad t\in[0,T),
\end{equation*}
where $\xi(t)\in\mathbb{R}$ is any point at which the function $u(\cdot,t)$ attains its maximum.
\end{theorem}

In addition, it is noticed that Eq. \eqref{mma} possesses at least three conservation laws defined by
\begin{equation}\label{mmb}
\begin{split}
&H_{0}[u]=\mu(u),\\
&H_{1}[u]=\frac{1}{2}\left[\int_{\mathbb{S}^{1}}u_{x}^{2}dx+\mu^{2}(u)\right],\\
&H_{2}[u]=\frac{1}{2}\int_{\mathbb{S}^{1}}\left[\sum\limits_{i=1}^{3}a_{i}u^{i}(2\mu(u)u
+u_{x}^{2})-\frac{23}{12}\mu^{2}(u)\sum\limits_{i=1}^{2}a_{i+1}u^{i+1}\right]dx,\\
\end{split}
\end{equation}
which play a key role in proving the orbital stability of peakons. While the corresponding three
useful conservation laws of the $\mu$-CH equation are in the following:
\begin{equation*}
F_{0}[u]=H_{0}[u],\quad\quad F_{1}[u]=H_{1}[u],\quad\quad
F_{2}[u]=\int_{\mathbb{S}^{1}}\left(\mu(u)u^{2}+\frac{1}{2}uu_{x}^{2}\right)dx.
\end{equation*}

The stability of periodic peaked solitons of the higher-order $\mu$-CH equation \eqref{mma}
is an interesting issue investigated in this paper. The approach of studying this issue is
inspired by the recent works \cite{cs,len,lqz}. It is crucial to establish a suitable inequality relating the maximum
and minimum of the perturbed solution with the conserved densities, which will depend on the
introduction of an appropriately constructed auxiliary function. Besides, the corresponding equality
needs to hold at the periodic peakons. This condition is very important because stable periodic
peakons must be critical points of the energy function with the momentum constrain, and satisfy
the corresponding Euler-Lagrangian equations.

However, we encounter some new difficulties. Firstly, it is found that the conservation law $H_{2}$ of
the higher-order $\mu$-CH equation is much more complex than $F_{2}$ of the $\mu$-CH equation.
Therefore, compared with the case of the $\mu$-CH equation in \cite{cll},
the stability issue of periodic peakons of the higher-order $\mu$-CH equation is more
subtle to deal with because the latter one involves more complicated calculations.
Secondly, the nonlinearity in Eq. \eqref{mma} is higher-order, which makes
the dynamical analysis more delicate.
Finally, the most difficult part is to construct two
suitable auxiliary functions $g(x)$ and $h(x)$, which related to three conservation laws
$H_{0}[u]$, $H_{1}[u]$ and $H_{2}[u]$. Further, we can obtain the real-valued function $E_{u}(M,m)$
with respect to the maximum and minimum of the perturbed solution with the conserved densities.
However, due to the complexity of the real-valued function $E_{u}(M,m)$, it is difficult to calculate
its partial derivatives up to the second order at the critical point $(M_{\varphi_{c}},m_{\varphi_{c}})$,
see Lemma \ref{lem3.4}.

The remainder of the paper is organized as follows. In Section 2, the well-posedness and the structure
of periodic peaked solitons of the higher-order $\mu$-CH equation are briefly reviewed on. Section 3 is devoted to proving that the periodic peakons of the higher-order $\mu$-CH equation are dynamically stable under small perturbations in the Sobolev space $H^{1}(\mathbb{S}^{1})$.\\

\noindent{\bf Notation.}
Throughout, the norm of a Banach space $Z$ is denoted by $\|\cdot\|_{Z}$. For the sake of simplicity,
we drop $\mathbb{S}^{1}$ in our notations of function spaces if there is no ambiguity because
all space of functions are over $\mathbb{S}^{1}$.

\renewcommand{\theequation}{\thesection.\arabic{equation}}
\setcounter{equation}{0}
\section{Preliminaries}

We are concerned with the following Cauchy problem associated with the equation \eqref{mma} on the unit circle $\mathbb{S}^{1}$, that is,
\begin{equation}\label{baa}
\begin{cases}
&u_{t}+\sum\limits_{i=1}^{3}a_{i}u^{i}u_{x}
+(\mu-\partial_{x}^{2})^{-1}\partial_{x}\Bigg(\sum\limits_{i=1}^{3}a_{i}u^{i-1}\left(\left(i+1\right)\mu(u)u+\frac{1}{2}iu_{x}^{2}\right)\\
&\quad-\frac{23}{24}\sum\limits_{i=1}^{2}(i+1)a_{i+1}\mu^{2}(u)u^{i}\Bigg)=0,\quad t>0,\quad x\in\mathbb{\mathbb{R}},\\
&u(0,x)=u_{0}(x),\quad\quad m=\mu(u)-u_{xx},\quad\quad x\in\mathbb{R},\\
&u(t,x+1)=u(t,x),\quad\quad t\geq0,\quad\quad x\in\mathbb{R}.
\end{cases}
\end{equation}

For an integer $n\geq 1$, we define $H^{n}$ as the Sobolev space of all square integrable functions
$f\in L^{2}$ with distributional derivatives $\partial_{x}^{i}f\in L^{2}$ for $i=1,...,n$.
The norm on $H^{n}$ is given by
\begin{equation*}
\|f\|_{H^{n}(\mathbb{S}^{1})}^{2}=\sum\limits_{i=0}^{n}\int_{\mathbb{S}^{1}}(\partial_{x}^{i}f)^{2}(x)dx.
\end{equation*}

First, we give the following definition of the strong solutions.

\noindent{\bf Definition 2.1}
If $u\in C([0,T); H^{s})\cap C^{1}([0,T); H^{s-1})$ with $s>\frac{3}{2}$ and some $T>0$ satisfies \eqref{baa},
then $u$ is called a strong solution of \eqref{baa} on $[0,T)$. If $u$ is a strong solution on $[0,T)$
for every $T>0$, then it is called a global strong solution.

We now recall the local well-posedness and properties for strong solutions of \eqref{baa} on
the unit circle proved in \cite{wlfq}.

\begin{lemma}\label{lem2.1} \cite{wlfq}
Let $u_{0}\in H^{s}$ with $s>3/2$. Then there exists $T>0$ such that the Cauchy
problem \eqref{baa} has a unique strong solution $u\in C([0,T),H^{s})\cap C^{1}([0,T),H^{s-1})$
and the map $u_{0}\mapsto u$ is continuous from a neighborhood of $u_{0}\in H^{s}$ into
$C([0,T),H^{s})\cap C^{1}([0,T),H^{s-1})$.
\end{lemma}

\begin{lemma}\label{lem2.2}
The Hamiltonian functionals \eqref{mmb} are conserved for the strong solution $u$ in Lemma \ref{lem2.1},
that is, for all $t\in [0,T)$
\begin{equation}\label{bab}
\begin{split}
&\frac{d}{dt}H_{0}[u]=\frac{d}{dt}\mu(u)=0,\\
&\frac{d}{dt}H_{1}[u]=\frac{1}{2}\frac{d}{dt}\left[\int_{\mathbb{S}^{1}}u_{x}^{2}dx+\mu^{2}(u)\right]=0,\\
&\frac{d}{dt}H_{2}[u]=\frac{1}{2}\frac{d}{dt}\int_{\mathbb{S}^{1}}\left[\sum\limits_{i=1}^{3}a_{i}u^{i}(2\mu(u)u
+u_{x}^{2})-\frac{23}{12}\mu^{2}(u)\sum\limits_{i=1}^{2}a_{i+1}u^{i+1}\right]dx=0.\\
\end{split}
\end{equation}
\end{lemma}
Besides, if $m_{0}(x)=(\mu-\partial_{x}^{2})u_{0}(x)$ does not change sign, then $m(t,x)$ will not change
sign for any $t\in[0,T)$. From $m_{0}(x)\geq 0$, we deduce that the corresponding solution $u(t,x)$ is
positive for $(t,x)\in[0,T)\times \mathbb{S}^{1}$.

Note that Eq. \eqref{mma} can also be equivalently rewritten as the fully nonlinear partial differential
equation as follows

\begin{equation}\label{bac}
\begin{split}
u_{t}&+\sum\limits_{i=1}^{3}a_{i}u^{i}u_{x}
+\phi_{x}\ast\Bigg(\sum\limits_{i=1}^{3}a_{i}u^{i-1}\left(\left(i+1\right)\mu(u)u+\frac{1}{2}iu_{x}^{2}\right)\\
&-\frac{23}{24}\sum\limits_{i=1}^{2}(i+1)a_{i+1}\mu^{2}(u)u^{i}\Bigg)=0.\\
\end{split}
\end{equation}
Recall that
\begin{equation*}
u=A^{-1}m=\phi\ast m,
\end{equation*}
where $\phi$ is the Green function of the operator $A=\mu-\partial_{x}^{2}$, defined by
\begin{equation*}
\phi(x)=\frac{1}{2}\left(x-\frac{1}{2}\right)^{2}+\frac{23}{24}.
\end{equation*}
Its derivative can be assigned to zero at $x=0$, so it follows that
\begin{align*}
\phi_{x}(x)=
\begin{cases}
 0,&\quad\quad x=0,\\
x-\frac{1}{2},&\quad\quad 0<x<1.
\end{cases}
\end{align*}

In fact, from the above formulation, we define the notion of weak solutions of \eqref{baa} as follows.

\noindent{\bf Definition 2.2}
Given initial data $u_{0}\in W^{1,4}$, the function $u\in L^{\infty}([0,T); W^{1,4})$ is
said to be a weak solution to the initial-value problem \eqref{baa} if it satisfies
the following identity:
\begin{equation*}
\begin{split}
\int_{0}^{T}&\int_{\mathbb{S}^{1}}\Bigg[-u\psi_{t}-\sum\limits_{i=1}^{3}\frac{1}{i+1}a_{i}u^{i+1}\psi_{x}
+\phi_{x}\ast\Bigg(\sum\limits_{i=1}^{3}a_{i}u^{i-1}\left(\left(i+1\right)\mu(u)u+\frac{1}{2}iu_{x}^{2}\right)\\
&-\frac{23}{24}\sum\limits_{i=1}^{2}(i+1)a_{i+1}\mu^{2}(u)u^{i}\Bigg)\psi\Bigg]dxdt
+\int_{\mathbb{S}^{1}}u_{0}(x)\psi(0,x)dx=0,\\
\end{split}
\end{equation*}
for any smooth test function $\psi(t,x)\in C_{c}^{\infty}([0,T)\times \mathbb{S}^{1})$. If $u$ is a weak solution
on $[0,T)$ for every $T>0$, then it is called a global weak solution.

\renewcommand{\theequation}{\thesection.\arabic{equation}}
\setcounter{equation}{0}
\section{Proof of Orbital Stability}
Our attention in this section is now turned to the issue of the stability of periodic peakons to the
higher-order $\mu$-CH equation \eqref{mma}. The proof of the stability of periodic peakons
is approached via a series of lemmas. The following lemma summarizes the properties of the periodic peakons.

\begin{lemma}\label{lem3.1}
The periodic peakon $\varphi_{c}(x)$ defined in \eqref{mmh}-\eqref{lblb} is continuous on $\mathbb{S}^{1}$ with peak
at $x=\pm\frac{1}{2}$. The extrema of $\varphi_{c}$ are
\begin{equation*}
M_{\varphi_{c}}=\max\limits_{x\in\mathbb{S}^{1}}\varphi_{c}(x)=\varphi_{c}\left(\frac{1}{2}\right)=\frac{13}{12}a=\frac{13}{12}H_{0}[\varphi_{c}],
\end{equation*}
\begin{equation*}
m_{\varphi_{c}}=\min\limits_{x\in\mathbb{S}^{1}}\varphi_{c}(x)=\varphi_{c}(0)=\frac{23}{24}a=\frac{23}{24}H_{0}[\varphi_{c}].
\end{equation*}
Moreover, we obtain
\begin{equation*}
\lim\limits_{x\uparrow\frac{1}{2}}\varphi_{c,x}(x)=\frac{1}{2}a=\frac{1}{2}H_{0}[\varphi_{c}],\quad
\lim\limits_{x\downarrow-\frac{1}{2}}\varphi_{c,x}(x)=-\frac{1}{2}a=-\frac{1}{2}H_{0}[\varphi_{c}]
\end{equation*}
and
\begin{equation*}
H_{0}[\varphi_{c}]=\mu(\varphi_{c})=a>0,\quad\quad
H_{1}[\varphi_{c}]=\frac{13}{24}a^{2}=\frac{13}{24}H_{0}^{2}[\varphi_{c}],
\end{equation*}
\begin{equation*}
\begin{split}
H_{2}[\varphi_{c}]&=\frac{47}{45}a^{3}a_{1}+\frac{5387}{60480}a^{4}a_{2}+\frac{16733}{181440}a^{5}a_{3}\\
&=\frac{47}{45}a_{1}H_{0}^{3}[\varphi_{c}]+\frac{5387}{60480}a_{2}H_{0}^{4}[\varphi_{c}]+\frac{16733}{181440}a_{3}H_{0}^{5}[\varphi_{c}].
\end{split}
\end{equation*}

\end{lemma}

\begin{proof}
The properties are easily derived from the expression of $\varphi_{c}(x)$
and the definition of $H_{0}$, $H_{1}$ and $H_{2}$. A direct calculation gives rise to
\begin{equation*}
H_{0}[\varphi_{c}]=\mu(\varphi_{c})=\int_{-\frac{1}{2}}^{\frac{1}{2}}a\left(\frac{1}{2}x^{2}+\frac{23}{24}\right)dx=a,
\end{equation*}

\begin{equation*}
\begin{split}
H_{1}[\varphi_{c}]&=\frac{1}{2}\left(\mu^{2}(\varphi_{c})+\int_{\mathbb{S}^{1}}\varphi_{c,x}^{2}(x)dx\right)
=\frac{1}{2}\left(a^{2}+\int_{-\frac{1}{2}}^{\frac{1}{2}}a^{2}x^{2}dx\right)\\
&=\frac{13}{24}a^{2}=\frac{13}{24}H_{0}^{2}[\varphi_{c}],\\
\end{split}
\end{equation*}

\begin{equation*}
\begin{split}
H_{2}[\varphi_{c}]&=\frac{1}{2}\int_{\mathbb{S}^{1}}\left[\sum\limits_{i=1}^{3}a_{i}\varphi_{c}^{i}
\left(2\mu(\varphi_{c})\varphi_{c}+\varphi_{c,x}^{2}\right)
-\frac{23}{12}\mu^{2}(\varphi_{c})\sum\limits_{i=1}^{2}a_{i+1}\varphi_{c}^{i+1}\right]dx\\
&=\left(a^{3}a_{1}-\frac{23}{24}a^{4}a_{2}\right)\int_{-\frac{1}{2}}^{\frac{1}{2}}\left(\frac{1}{4}x^{4}+\frac{23}{24}x^{2}+\frac{23^{2}}{24^{2}}\right)dx\\
&\quad+\left(a^{4}a_{2}-\frac{23}{24}a^{5}a_{3}\right)\int_{-\frac{1}{2}}^{\frac{1}{2}}
\left(\frac{1}{8}x^{6}+\frac{23}{32}x^{4}+\frac{3}{2}\times\frac{23^{2}}{24^{2}}x^{2}+\frac{23^{3}}{24^{3}}\right)dx\\
&\quad+\frac{1}{2}a^{3}a_{1}\int_{-\frac{1}{2}}^{\frac{1}{2}}\left(\frac{1}{2}x^{4}+\frac{23}{24}x^{2}\right)dx
+\frac{1}{2}a^{4}a_{2}\int_{-\frac{1}{2}}^{\frac{1}{2}}\left(\frac{1}{4}x^{6}+\frac{23}{24}x^{4}+\frac{23^{2}}{24^{2}}x^{2}\right)dx\\
&\quad+a^{5}a_{3}\int_{-\frac{1}{2}}^{\frac{1}{2}}\left(\frac{1}{8}x^{8}+\left(\frac{23}{48}+\frac{23}{64}\right)x^{6}
+\frac{9}{4}\times\frac{23^{2}}{24^{2}}x^{4}+\frac{5}{2}\times\frac{23^{3}}{24^{3}}x^{2}+\frac{23^{4}}{24^{4}}\right)dx\\
&=\frac{47}{45}a^{3}a_{1}+\frac{5387}{60480}a^{4}a_{2}+\frac{16733}{181440}a^{5}a_{3}\\
&=\frac{47}{45}a_{1}H_{0}^{3}[\varphi_{c}]+\frac{5387}{60480}a_{2}H_{0}^{4}[\varphi_{c}]+\frac{16733}{181440}a_{3}H_{0}^{5}[\varphi_{c}].\\
\end{split}
\end{equation*}
\end{proof}

We define the $\mu$-inner product $\langle\cdot,\cdot\rangle_{\mu}$ and the related $\mu$-norm $\|\cdot\|_{\mu}$ by
\begin{equation}\label{caa}
\begin{split}
&\langle u,v\rangle_{\mu}=\mu(u)\mu(v)+\int_{\mathbb{S}^{1}}u_{x}v_{x}dx,\\
&\|u\|_{\mu}^{2}=\langle u,u\rangle_{\mu}=2H_{1}[u],\quad\quad
u,v\in H^{1}(\mathbb{S}^{1}),\\
\end{split}
\end{equation}
and consider the expansion of the conservation law $H_{1}$ around the periodic peakon $\varphi_{c}$ in
the $\mu$-norm. In the following lemma, we will show that the error term in this expansion is given by
the product of $a$ and the difference between $\varphi_{c}$ and the perturbed solution $u$ at the point of the peak.

\begin{lemma}\label{lem3.2}
For every $u\in H^{1}(\mathbb{S}^{1})$ and $\xi\in\mathbb{R}$,
\begin{equation*}
H_{1}[u]-H_{1}[\varphi_{c}]=\frac{1}{2}\|u-\varphi_{c}\left(\cdot-\xi\right)\|_{\mu}^{2}
+a\left(u\left(\xi+\frac{1}{2}\right)-M_{\varphi_{c}}\right).
\end{equation*}
\end{lemma}

\begin{proof}
A direct computation gives
\begin{equation*}
\begin{split}
\frac{1}{2}\|u-\varphi_{c}\left(\cdot-\xi\right)\|_{\mu}^{2}
&=H_{1}[u]+H_{1}[\varphi_{c}(\cdot-\xi)]-\mu(u)\mu(\varphi_{c})-\int_{\mathbb{S}^{1}}u_{x}(x)\varphi_{c,x}(x-\xi)dx\\
&=H_{1}[u]+H_{1}[\varphi_{c}]-\mu(u)\mu(\varphi_{c})+\int_{\mathbb{S}^{1}}u(x+\xi)\varphi_{c,xx}(x)dx.\\
\end{split}
\end{equation*}
Since
\begin{equation}\label{cab}
\varphi_{c,xx}=a\left[1-\delta\left(x-\frac{1}{2}\right)\right],
\end{equation}
it follows that
\begin{equation*}
\int_{\mathbb{S}^{1}}u(x+\xi)\varphi_{c,xx}(x)dx=a\int_{\mathbb{S}^{1}}u(x)dx-au\left(\xi+\frac{1}{2}\right).
\end{equation*}
Applying $H_{0}[\varphi_{c}]=\mu(\varphi_{c})=a$, we deduce that
\begin{equation*}
\frac{1}{2}\|u-\varphi_{c}\left(\cdot-\xi\right)\|_{\mu}^{2}
=H_{1}[u]-H_{1}[\varphi_{c}]+a\left(\frac{13}{12}a-u\left(\xi+\frac{1}{2}\right)\right),
\end{equation*}
which completes the proof of Lemma \ref{lem3.2}.

\end{proof}
\begin{remark}\label{rem3.1}
For a wave profile $u\in H^{1}(\mathbb{S}^{1})$, the functional $H_{1}[u]$ stands for kinetic energy.
It follows from Lemma \ref{lem3.2} that if a wave $u\in H^{1}(\mathbb{S}^{1})$ has energy $H_{1}[u]$ and
height $M_{u}$ close to the peakon's energy and height, then the whole shape of $u$ is close to that
of the peakon. We can also derive from Lemma \ref{lem3.2} that the peakon has maximal height among all waves
of fixed energy. Indeed, if $u\in H^{1}(\mathbb{S}^{1})\subset C(\mathbb{S}^{1})$ is such that
$H_{1}[u]=H_{1}[\varphi_{c}]$ and $u(\xi)=\max_{x\in\mathbb{S}^{1}}u(x)$, then $u(\xi)\leq M_{\varphi_{c}}$.
The remarks are also the same with the $\mu$-CH, the modified $\mu$-CH, the generalized $\mu$-CH, the CH and the
modified CH equations \cite{cll,len,lqz,qll,qzll}.

\end{remark}

The following lemma plays a crucial role in establishing the result of the stability of periodic peakons.

\begin{lemma}\label{lem3.3}
For any function $u\in H^{1}(\mathbb{S}^{1})$ with $\mu(u)>0$, define the function
\begin{equation*}
E_{u}: \left\{(M,m)\in\mathbb{R}^{2}: M\geq m>0\right\}\rightarrow \mathbb{R},
\end{equation*}
by
\begin{equation*}
\begin{split}
&E_{u}(M,m)\\
&=\left(a_{1}M+a_{2}M^{2}+a_{3}M^{3}\right)\Bigg[H_{1}[u]+\frac{1}{2}H_{0}^{2}[u]
-mH_{0}[u]\\&\qquad\qquad\qquad\qquad\qquad\qquad -\frac{2}{3H_{0}[u]}\left(2H_{0}[u](M-m)\right)^{\frac{3}{2}}\Bigg]+a_{1}mH_{0}^{2}[u]-H_{2}[u]\\
&\quad+\left(H_{0}[u]m-\frac{23}{24}H_{0}^{2}[u]\right)a_{2}\int_{\mathbb{S}^{1}}u^{2}dx
 +\left(H_{0}[u]m-\frac{23}{24}H_{0}^{2}[u]\right)a_{3}\int_{\mathbb{S}^{1}}u^{3}dx\\
&\quad+\frac{2}{H_{0}[u]}\Bigg(\frac{2a_{1}}{15}m+\frac{a_{1}}{5}M+\frac{a_{2}}{7}M^{2}
+\frac{4a_{2}}{35}mM\\&\qquad\qquad+\frac{8a_{2}}{105}m^{2}+\frac{a_{3}}{9}M^{3}+\frac{2a_{3}}{21}mM^{2}+\frac{8a_{3}}{105}m^{2}M+\frac{16a_{3}}{315}m^{3}\Bigg)\left(2H_{0}[u](M-m)\right)^{\frac{3}{2}}.\\
\end{split}
\end{equation*}
Then
\begin{equation*}
E_{u}\left(M_{u},m_{u}\right)\geq 0,
\end{equation*}
where $M_{u}=\max_{x\in\mathbb{S}^{1}}\{u(x)\}$ and $m_{u}=\min_{x\in\mathbb{S}^{1}}\{u(x)\}$.
\end{lemma}
\begin{remark}\label{rem3.2}
It is obvious that the function $E_{u}$ depends on $u$ only through three conservation laws $H_{0}[u]$,
$H_{1}[u]$, $H_{2}[u]$, the $L^{2}$-norm and $L^{3}$-norm of u.
\end{remark}
\begin{proof}
Note that the peakon $\varphi_{c}$ satisfies the following differential equation
\begin{align*}
\partial_{x}\varphi_{c}(x)=
\begin{cases}
-\sqrt{2\mu(\varphi_{c})(\varphi_{c}-m_{\varphi_{c}})},&\quad\quad -\frac{1}{2}<x\leq0,\\
\sqrt{2\mu(\varphi_{c})(\varphi_{c}-m_{\varphi_{c}})},&\quad\quad 0\leq x<\frac{1}{2}.
\end{cases}
\end{align*}
Let $u\in H^{1}(\mathbb{S}^{1})\subset C(\mathbb{S}^{1})$ with $\mu(u)>0$. Define
$M=M_{u}=\max_{x\in\mathbb{S}^{1}}u(x)$, $m=m_{u}=\min_{x\in\mathbb{S}^{1}}u(x)$.
Let $\xi$ and $\eta$ be such that $u(\xi)=M$ and $u(\eta)=m$. We first define the real-valued function
$g(x)$ by
\begin{align}\label{cac}
g(x)=
\begin{cases}
u_{x}+\sqrt{2\mu(u)(u-m)},&\quad\quad \xi<x\leq \eta,\\
u_{x}-\sqrt{2\mu(u)(u-m)},&\quad\quad \eta\leq x<\xi+1,
\end{cases}
\end{align}
and extend it periodically to the real line. A direct calculation leads to
\begin{equation}\label{cad}
\frac{1}{2}\int_{\mathbb{S}^{1}}g^{2}(x)dx=H_{1}[u]+\frac{1}{2}H_{0}^{2}[u]-mH_{0}[u]-\frac{2}{3H_{0}[u]}\left(2H_{0}[u](M-m)\right)^{\frac{3}{2}}.
\end{equation}
On the other hand, we introduce the real-valued function $h(x)$ defined by
\begin{equation}\label{lala}
h(x)=a_{1}u+a_{2}u^{2}+a_{3}u^{3},\quad\quad \xi<x<\xi+1,
\end{equation}
and extend it periodically to the entire real line. It is inferred that
\begin{align}\label{caf}
&\int_{\mathbb{S}^{1}}\left(a_{1}u+a_{2}u^{2}+a_{3}u^{3}\right)g^{2}(x)dx\nonumber\\
&=\int_{\xi}^{\eta}\left(a_{1}u+a_{2}u^{2}+a_{3}u^{3}\right)\left(u_{x}^{2}+2\mu(u)(u-m)+2u_{x}\sqrt{2\mu(u)(u-m)}\right)dx\nonumber\\
&\quad+\int_{\eta}^{\xi+1}\left(a_{1}u+a_{2}u^{2}+a_{3}u^{3}\right)\left(u_{x}^{2}+2\mu(u)(u-m)-2u_{x}\sqrt{2\mu(u)(u-m)}\right)dx\nonumber\\
&=I_{1}+I_{2}.
\end{align}
By a simple computation, we deduce that

\begin{align}\label{cae}
I_{1}&=\int_{\xi}^{\eta}\bigg(a_{1}uu_{x}^{2}+a_{2}u^{2}u_{x}^{2}+a_{3}u^{3}u_{x}^{2}
+2a_{1}\mu(u)u^{2}+2a_{2}\mu(u)u^{3}+2a_{3}\mu(u)u^{4}\nonumber\\
&\quad-\frac{23}{12}a_{2}\mu^{2}(u)u^{2}-\frac{23}{12}a_{3}\mu^{2}(u)u^{3}\bigg)dx
-2a_{1}\mu(u)m\int_{\xi}^{\eta}udx\nonumber\\
&\quad+\left(\frac{23}{12}\mu^{2}(u)-2\mu(u)m\right)a_{2}\int_{\xi}^{\eta}u^{2}dx
+\left(\frac{23}{12}\mu^{2}(u)-2\mu(u)m\right)a_{3}\int_{\xi}^{\eta}u^{3}dx\nonumber\\
&\quad+2a_{1}\int_{\xi}^{\eta}uu_{x}\sqrt{2\mu(u)(u-m)}dx
+2a_{2}\int_{\xi}^{\eta}u^{2}u_{x}\sqrt{2\mu(u)(u-m)}dx\nonumber\\
&\quad+2a_{3}\int_{\xi}^{\eta}u^{3}u_{x}\sqrt{2\mu(u)(u-m)}dx.
\end{align}
Notice that
\begin{equation*}
\frac{d}{dx}\left[\frac{2a_{1}}{5\mu(u)}\left(2\mu(u)(u-m)\right)^{\frac{3}{2}}\left(\frac{2}{3}m+u\right)\right]
=2a_{1}uu_{x}\sqrt{2\mu(u)(u-m)},
\end{equation*}

\begin{equation*}
\begin{split}
&\frac{d}{dx}\left[2a_{2}\left(\frac{1}{7\mu(u)}u^{2}+\frac{4m}{35\mu(u)}u
+\frac{8m^{2}}{105\mu(u)}\right)\left(2\mu(u)(u-m)\right)^{\frac{3}{2}}\right]\\
&=2a_{2}u^{2}u_{x}\sqrt{2\mu(u)(u-m)}\\
\end{split}
\end{equation*}
and
\begin{equation*}
\begin{split}
&\frac{d}{dx}\left[2a_{2}\left(\frac{1}{9\mu(u)}u^{3}+\frac{2m}{21\mu(u)}u^{2}+\frac{8m^{2}}{105\mu(u)}u
+\frac{16m^{3}}{315\mu(u)}\right)\left(2\mu(u)(u-m)\right)^{\frac{3}{2}}\right]\\
&=2a_{3}u^{3}u_{x}\sqrt{2\mu(u)(u-m)},\\
\end{split}
\end{equation*}
we derive that
\begin{equation*}
\begin{split}
2a_{1}\int_{\xi}^{\eta}uu_{x}\sqrt{2\mu(u)(u-m)}dx
&=\left[\frac{2a_{1}}{5\mu(u)}\left(2\mu(u)(u-m)\right)^{\frac{3}{2}}\left(\frac{2}{3}m+u\right)\right]\bigg|_{\xi}^{\eta}\\
&=-\frac{2a_{1}}{5\mu(u)}\left(2\mu(u)(M-m)\right)^{\frac{3}{2}}\left(\frac{2}{3}m+M\right),\\
\end{split}
\end{equation*}

\begin{equation*}
\begin{split}
&2a_{2}\int_{\xi}^{\eta}u^{2}u_{x}\sqrt{2\mu(u)(u-m)}dx\\
&=\left[2a_{2}\left(\frac{1}{7\mu(u)}u^{2}+\frac{4m}{35\mu(u)}u
+\frac{8m^{2}}{105\mu(u)}\right)\left(2\mu(u)(u-m)\right)^{\frac{3}{2}}\right]\bigg|_{\xi}^{\eta}\\
&=-2a_{2}\left(\frac{1}{7\mu(u)}M^{2}+\frac{4m}{35\mu(u)}M
+\frac{8m^{2}}{105\mu(u)}\right)\left(2\mu(u)(M-m)\right)^{\frac{3}{2}}\\
\end{split}
\end{equation*}
and
\begin{equation*}
\begin{split}
&2a_{3}\int_{\xi}^{\eta}u^{3}u_{x}\sqrt{2\mu(u)(u-m)}dx\\
&=\left[2a_{3}\left(\frac{1}{9\mu(u)}u^{3}+\frac{2m}{21\mu(u)}u^{2}+\frac{8m^{2}}{105\mu(u)}u
+\frac{16m^{3}}{315\mu(u)}\right)\left(2\mu(u)(u-m)\right)^{\frac{3}{2}}\right]\bigg|_{\xi}^{\eta}\\
&=-2a_{3}\left(\frac{1}{9\mu(u)}M^{3}+\frac{2m}{21\mu(u)}M^{2}+\frac{8m^{2}}{105\mu(u)}M
+\frac{16m^{3}}{315\mu(u)}\right)\left(2\mu(u)(M-m)\right)^{\frac{3}{2}}.\\
\end{split}
\end{equation*}
Putting together the above identities into \eqref{cae}, it follows that
\begin{align}\label{wfa}
I_{1}&=\int_{\xi}^{\eta}\bigg(a_{1}uu_{x}^{2}+a_{2}u^{2}u_{x}^{2}+a_{3}u^{3}u_{x}^{2}
+2a_{1}\mu(u)u^{2}+2a_{2}\mu(u)u^{3}+2a_{3}\mu(u)u^{4}\nonumber\\
&\quad\quad\quad\quad-\frac{23}{12}a_{2}\mu^{2}(u)u^{2}-\frac{23}{12}a_{3}\mu^{2}(u)u^{3}\bigg)dx
-2a_{1}\mu(u)m\int_{\xi}^{\eta}udx\nonumber\\
&\quad+\left(\frac{23}{12}\mu^{2}(u)-2\mu(u)m\right)a_{2}\int_{\xi}^{\eta}u^{2}dx
+\left(\frac{23}{12}\mu^{2}(u)-2\mu(u)m\right)a_{3}\int_{\xi}^{\eta}u^{3}dx\nonumber\\
&\quad-\frac{2}{\mu(u)}\bigg(\frac{2a_{1}}{15}m+\frac{a_{1}}{5}M+\frac{a_{2}}{7}M^{2}+\frac{4a_{2}}{35}mM
+\frac{8a_{2}}{105}m^{2}+\frac{a_{3}}{9}M^{3}+\frac{2a_{3}}{21}mM^{2}\nonumber\\
&\quad\quad\quad\quad\quad+\frac{8a_{3}}{105}m^{2}M+\frac{16a_{3}}{315}m^{3}\bigg)(2\mu(u)(M-m))^{\frac{3}{2}}.
\end{align}
In a similar way, one obtains
\begin{align}\label{wfb}
I_{2}&=\int_{\eta}^{\xi+1}\bigg(a_{1}uu_{x}^{2}+a_{2}u^{2}u_{x}^{2}+a_{3}u^{3}u_{x}^{2}
+2a_{1}\mu(u)u^{2}+2a_{2}\mu(u)u^{3}+2a_{3}\mu(u)u^{4}\nonumber\\
&\quad\quad\quad\quad\quad-\frac{23}{12}a_{2}\mu^{2}(u)u^{2}-\frac{23}{12}a_{3}\mu^{2}(u)u^{3}\bigg)dx
-2a_{1}\mu(u)m\int_{\eta}^{\xi+1}udx\nonumber\\
&\quad+\left(\frac{23}{12}\mu^{2}(u)-2\mu(u)m\right)a_{2}\int_{\eta}^{\xi+1}u^{2}dx
+\left(\frac{23}{12}\mu^{2}(u)-2\mu(u)m\right)a_{3}\int_{\eta}^{\xi+1}u^{3}dx\nonumber\\
&\quad-\frac{2}{\mu(u)}\bigg(\frac{2a_{1}}{15}m+\frac{a_{1}}{5}M+\frac{a_{2}}{7}M^{2}+\frac{4a_{2}}{35}mM
+\frac{8a_{2}}{105}m^{2}+\frac{a_{3}}{9}M^{3}+\frac{2a_{3}}{21}mM^{2}\nonumber\\
&\quad\quad\quad\quad\quad+\frac{8a_{3}}{105}m^{2}M+\frac{16a_{3}}{315}m^{3}\bigg)(2\mu(u)(M-m))^{\frac{3}{2}}.
\end{align}
Substituting \eqref{wfa} and \eqref{wfb} into \eqref{caf} and using \eqref{mmb}, we infer that
\begin{equation}\label{cag}
\begin{split}
&\frac{1}{2}\int_{\mathbb{S}^{1}}\left(a_{1}u+a_{2}u^{2}+a_{3}u^{3}\right)g^{2}(x)dx\\
&=H_{2}[u]-a_{1}H_{0}^{2}[u]m+\left(\frac{23}{24}H_{0}^{2}[u]-H_{0}[u]m\right)a_{2}\int_{\mathbb{S}^{1}}u^{2}dx\\
&\quad+\left(\frac{23}{24}H_{0}^{2}[u]-H_{0}[u]m\right)a_{3}\int_{\mathbb{S}^{1}}u^{3}dx\\
&\quad-\frac{2}{H_{0}[u]}\bigg(\frac{2a_{1}}{15}m+\frac{a_{1}}{5}M+\frac{a_{2}}{7}M^{2}
+\frac{4a_{2}}{35}mM+\frac{8a_{2}}{105}m^{2}+\frac{a_{3}}{9}M^{3}\\
&\quad\quad\quad\quad\quad\quad+\frac{2a_{3}}{21}mM^{2}
+\frac{8a_{3}}{105}m^{2}M+\frac{16a_{3}}{315}m^{3}\bigg)(2H_{0}[u](M-m))^{\frac{3}{2}}.\\
\end{split}
\end{equation}
Combining \eqref{cad} with \eqref{cag} implies
\begin{align*}
H_{2}[u]&=\frac{1}{2}\int_{\mathbb{S}^{1}}\left(a_{1}u+a_{2}u^{2}+a_{3}u^{3}\right)g^{2}(x)dx+a_{1}H_{0}^{2}[u]m\nonumber\\
&\quad+\left(H_{0}[u]m-\frac{23}{24}H_{0}^{2}[u]\right)a_{2}\int_{\mathbb{S}^{1}}u^{2}dx
+\left(H_{0}[u]m-\frac{23}{24}H_{0}^{2}[u]\right)a_{3}\int_{\mathbb{S}^{1}}u^{3}dx\nonumber\\
&\quad+\frac{2}{H_{0}[u]}\bigg(\frac{2a_{1}}{15}m+\frac{a_{1}}{5}M+\frac{a_{2}}{7}M^{2}+\frac{4a_{2}}{35}mM
+\frac{8a_{2}}{105}m^{2}+\frac{a_{3}}{9}M^{3}+\frac{2a_{3}}{21}mM^{2}\nonumber\\
&\quad\quad\quad\quad\quad\quad+\frac{8a_{3}}{105}m^{2}M+\frac{16a_{3}}{315}m^{3}\bigg)(2H_{0}[u](M-m))^{\frac{3}{2}}\nonumber\\
&\leq\frac{1}{2}\left(a_{1}M+a_{2}M^{2}+a_{3}M^{3}\right)\int_{\mathbb{S}^{1}}g^{2}(x)dx+a_{1}H_{0}^{2}[u]m\nonumber\\
&\quad+\left(H_{0}[u]m-\frac{23}{24}H_{0}^{2}[u]\right)a_{2}\int_{\mathbb{S}^{1}}u^{2}dx
+\left(H_{0}[u]m-\frac{23}{24}H_{0}^{2}[u]\right)a_{3}\int_{\mathbb{S}^{1}}u^{3}dx\nonumber\\
&\quad+\frac{2}{H_{0}[u]}\bigg(\frac{2a_{1}}{15}m+\frac{a_{1}}{5}M+\frac{a_{2}}{7}M^{2}+\frac{4a_{2}}{35}mM
+\frac{8a_{2}}{105}m^{2}+\frac{a_{3}}{9}M^{3}+\frac{2a_{3}}{21}mM^{2}\nonumber\\
&\quad\quad\quad\quad\quad\quad+\frac{8a_{3}}{105}m^{2}M+\frac{16a_{3}}{315}m^{3}\bigg)(2H_{0}[u](M-m))^{\frac{3}{2}}\nonumber\\
&=\left(a_{1}M+a_{2}M^{2}+a_{3}M^{3}\right)\bigg[H_{1}[u]+\frac{1}{2}H_{0}^{2}[u]-mH_{0}[u]\nonumber\\
\end{align*}
\begin{align}\label{cah}
&\quad\quad\quad\quad\quad\quad\quad\quad\quad\quad\quad\quad\quad\quad\quad
-\frac{2}{3H_{0}[u]}\left(2H_{0}[u](M-m)\right)^{\frac{3}{2}}\bigg]+a_{1}H_{0}^{2}[u]m\nonumber\\
&\quad\quad\quad\quad+\left(H_{0}[u]m-\frac{23}{24}H_{0}^{2}[u]\right)a_{2}\int_{\mathbb{S}^{1}}u^{2}dx
+\left(H_{0}[u]m-\frac{23}{24}H_{0}^{2}[u]\right)a_{3}\int_{\mathbb{S}^{1}}u^{3}dx\nonumber\\
&\quad\quad\quad\quad+\frac{2}{H_{0}[u]}\bigg(\frac{2a_{1}}{15}m+\frac{a_{1}}{5}M+\frac{a_{2}}{7}M^{2}+\frac{4a_{2}}{35}mM
+\frac{8a_{2}}{105}m^{2}+\frac{a_{3}}{9}M^{3}+\frac{2a_{3}}{21}mM^{2}\nonumber\\
&\quad\quad\quad\quad\quad\quad\quad\quad\quad+\frac{8a_{3}}{105}m^{2}M+\frac{16a_{3}}{315}m^{3}\bigg)(2H_{0}[u](M-m))^{\frac{3}{2}}.
\end{align}
It is obvious that
\begin{equation*}
E_{u}(M,m)\geq0.
\end{equation*}
This completes the proof of Lemma \ref{lem3.3}.
\end{proof}

In the next lemma, we highlight some properties of the function $E_{u}(M,m)$ related to the peakon $\varphi_{c}$.
\begin{lemma}\label{lem3.4}
Assume that $a>0$ and $a_{i}$, $i=1,2,3$ satisfy one of the following five conditions

$(i)$ $a_{1}>0$, $a_{2}>0$, $a_{3}>0$.

$(ii)$ $a_{2}<0$, $a_{2}^{2}<3a_{1}a_{3}$, $a_{3}>0$.

$(iii)$ $a_{1}>0$, $a_{2}>-\frac{6a_{1}}{13a}$, $a_{3}=0$.

$(iv)$ $a_{1}=0$, $a_{2}>0$, $a_{3}=0$.

$(v)$ $a_{1}<0$, $a_{2}>-\frac{a_{1}}{2a}$, $a_{3}=0$,

\noindent with the wave speed $c$ satisfies \eqref{mmi}.
Then, for the periodic peakon $\varphi_{c}$, it holds that
\begin{equation*}
\begin{split}
&E_{\varphi_{c}}\left(M_{\varphi_{c}}, m_{\varphi_{c}}\right)=0,\quad\quad
\frac{\partial E_{\varphi_{c}}}{\partial M}\left(M_{\varphi_{c}}, m_{\varphi_{c}}\right)=0,\\
&\frac{\partial E_{\varphi_{c}}}{\partial m}\left(M_{\varphi_{c}}, m_{\varphi_{c}}\right)=0,\quad\quad
\frac{\partial^{2} E_{\varphi_{c}}}{\partial M^{2}}\left(M_{\varphi_{c}}, m_{\varphi_{c}}\right)
=-aa_{1}-\frac{13}{6}a^{2}a_{2}-\frac{169}{48}a^{3}a_{3},\\
&\frac{\partial^{2} E_{\varphi_{c}}}{\partial M\partial m}\left(M_{\varphi_{c}}, m_{\varphi_{c}}\right)=0,\quad\quad
\frac{\partial^{2} E_{\varphi_{c}}}{\partial m^{2}}\left(M_{\varphi_{c}}, m_{\varphi_{c}}\right)
=-aa_{1}-2a^{2}a_{2}-\frac{721}{240}a^{3}a_{3}.\\
\end{split}
\end{equation*}
Moreover, $(M_{\varphi_{c}}, m_{\varphi_{c}})$ is an isolated local maximum of $E_{\varphi_{c}}$.

\end{lemma}

\begin{proof}
Thanks to the fact that corresponding to the periodic peakon $\varphi_{c}$, $g(x)$ is identical to zero,
we derive that $E_{\varphi_{c}}(M_{\varphi_{c}}, m_{\varphi_{c}})=0$. On the other hand, from Lemma
\ref{lem3.1}, it is inferred that

\begin{equation*}
M_{\varphi_{c}}=\frac{13}{12}a=\frac{13}{12}H_{0}[\varphi_{c}],\quad
m_{\varphi_{c}}=\frac{23}{24}a=\frac{23}{24}H_{0}[\varphi_{c}],
\end{equation*}
\begin{equation*}
H_{1}[\varphi_{c}]=\frac{13}{24}a^{2}=\frac{13}{24}H_{0}^{2}[\varphi_{c}],\quad\quad
2\mu(\varphi_{c})\left(M_{\varphi_{c}}-m_{\varphi_{c}}\right)=\frac{1}{4}a^{2}=\frac{1}{4}H_{0}^{2}[\varphi_{c}],
\end{equation*}
\begin{equation*}
\sqrt{2\mu(\varphi_{c})\left(M_{\varphi_{c}}-m_{\varphi_{c}}\right)}=\frac{1}{2}a=\frac{1}{2}H_{0}[\varphi_{c}].
\end{equation*}
From Lemma \ref{lem3.3}, we directly differentiate $E_{u}(M,m)$ with respect to $M$ and $m$ as following:
\begin{equation*}
\begin{split}
\frac{\partial E_{u}}{\partial M}
&=\left(a_{1}+2a_{2}M+3a_{3}M^{2}\right)\left(H_{1}[u]+\frac{1}{2}H_{0}^{2}[u]-H_{0}[u]m\right)\\
&\quad+\bigg(\frac{4a_{1}}{5}m-\frac{4a_{1}}{5}M-\frac{8a_{2}}{7}M^{2}+\frac{24a_{2}}{35}mM+\frac{16a_{2}}{35}m^{2}
-\frac{4a_{3}}{3}M^{3}\\
&\quad\quad\quad+\frac{4a_{3}}{7}mM^{2}+\frac{16a_{3}}{35}m^{2}M+\frac{32a_{3}}{105}m^{3}\bigg)(2H_{0}[u](M-m))^{\frac{1}{2}}\\
\end{split}
\end{equation*}
\begin{equation*}
\begin{split}
&\quad+\frac{2}{H_{0}[u]}\bigg(-\frac{2a_{1}}{15}-\frac{8a_{2}}{21}M+\frac{4a_{2}}{35}m
-\frac{2a_{3}}{3}M^{2}+\frac{4a_{3}}{21}mM\\
&\quad\quad\quad\quad\quad\quad+\frac{8a_{3}}{105}m^{2}\bigg)\left(2H_{0}[u](M-m)\right)^{\frac{3}{2}},\\
\end{split}
\end{equation*}

\begin{equation*}
\begin{split}
\frac{\partial E_{u}}{\partial m}
&=\left(a_{1}M+a_{2}M^{2}+a_{3}M^{3}\right)\left[-H_{0}[u]+2(2H_{0}[u](M-m))^{\frac{1}{2}}\right]\\
&\quad+a_{1}H_{0}^{2}[u]+a_{2}H_{0}[u]\int_{\mathbb{S}^{1}}u^{2}dx+a_{3}H_{0}[u]\int_{\mathbb{S}^{1}}u^{3}dx\\
&\quad-6\bigg(\frac{2a_{1}}{15}m+\frac{a_{1}}{5}M+\frac{a_{2}}{7}M^{2}+\frac{4a_{2}}{35}mM+\frac{8a_{2}}{105}m^{2}
+\frac{a_{3}}{9}M^{3}\\
&\quad\quad\quad\quad+\frac{2a_{3}}{21}mM^{2}+\frac{8a_{3}}{105}m^{2}M+\frac{16a_{3}}{315}m^{3}\bigg)(2H_{0}[u](M-m))^{\frac{1}{2}}\\
&\quad+\frac{2}{H_{0}[u]}\bigg(\frac{2a_{1}}{15}+\frac{4a_{2}}{35}M+\frac{16a_{2}}{105}m
+\frac{2a_{3}}{21}M^{2}+\frac{16a_{3}}{105}mM\\
&\quad\quad\quad\quad\quad\quad+\frac{16a_{3}}{105}m^{2}\bigg)\left(2H_{0}[u](M-m)\right)^{\frac{3}{2}},\\
\end{split}
\end{equation*}

\begin{equation*}
\begin{split}
\frac{\partial^{2} E_{u}}{\partial M^{2}}
&=\left(2a_{2}+6a_{3}M\right)\left(H_{1}[u]+\frac{1}{2}H_{0}^{2}[u]-H_{0}[u]m\right)\\
&\quad+2H_{0}[u]\bigg(\frac{2a_{1}}{5}m-\frac{2a_{1}}{5}M-\frac{4a_{2}}{7}M^{2}+\frac{12a_{2}}{35}mM
+\frac{8a_{2}}{35}m^{2}-\frac{2a_{3}}{3}M^{3}\\
&\quad\quad\quad\quad\quad\quad+\frac{2a_{3}}{7}mM^{2}+\frac{8a_{3}}{35}m^{2}M
+\frac{16a_{3}}{105}m^{3}\bigg)\left(2H_{0}[u](M-m)\right)^{-\frac{1}{2}}\\
&\quad+4\times\bigg(-\frac{2a_{1}}{5}-\frac{8a_{2}}{7}M+\frac{12a_{2}}{35}m-2a_{3}M^{2}+\frac{4a_{3}}{7}mM\\
&\quad\quad\quad\quad\quad+\frac{8a_{3}}{35}m^{2}\bigg)\left(2H_{0}[u](M-m)\right)^{\frac{1}{2}}\\
&\quad-\frac{2}{H_{0}[u]}\left(\frac{8a_{2}}{21}+\frac{4a_{3}}{3}M
-\frac{4a_{3}}{21}m\right)\left(2H_{0}[u](M-m)\right)^{\frac{3}{2}},\\
\end{split}
\end{equation*}

\begin{equation*}
\begin{split}
\frac{\partial^{2} E_{u}}{\partial M\partial m}
&=-H_{0}[u](a_{1}+2a_{2}M+3a_{3}M^{2})\\
&\quad-2H_{0}[u]\bigg(\frac{2a_{1}}{5}m-\frac{2a_{1}}{5}M-\frac{4a_{2}}{7}M^{2}+\frac{12a_{2}}{35}mM
+\frac{8a_{2}}{35}m^{2}-\frac{2a_{3}}{3}M^{3}\\
&\quad\quad\quad\quad\quad\quad+\frac{2a_{3}}{7}mM^{2}+\frac{8a_{3}}{35}m^{2}M
+\frac{16a_{3}}{105}m^{3}\bigg)\left(2H_{0}[u](M-m)\right)^{-\frac{1}{2}}\\
&\quad+2\bigg(\frac{4a_{1}}{5}+\frac{52a_{2}}{35}M+\frac{4a_{2}}{35}m+\frac{16a_{3}}{7}M^{2}
-\frac{4a_{3}}{35}mM\\
&\quad\quad\quad\quad+\frac{8a_{3}}{35}m^{2}\bigg)(2H_{0}[u](M-m))^{\frac{1}{2}}\\
&\quad+\frac{2}{H_{0}[u]}\left(\frac{4a_{2}}{35}+\frac{4a_{3}}{21}M+\frac{16a_{3}}{105}m\right)
(2H_{0}[u](M-m))^{\frac{3}{2}}\\
\end{split}
\end{equation*}
and
\begin{equation*}
\begin{split}
\frac{\partial^{2} E_{u}}{\partial m^{2}}
&=2H_{0}[u]\bigg(\frac{2a_{1}}{5}m-\frac{2a_{1}}{5}M-\frac{4a_{2}}{7}M^{2}+\frac{12a_{2}}{35}mM
+\frac{8a_{2}}{35}m^{2}-\frac{2a_{3}}{3}M^{3}\\
&\quad\quad\quad\quad\quad+\frac{2a_{3}}{7}mM^{2}+\frac{8a_{3}}{35}m^{2}M
+\frac{16a_{3}}{105}m^{3}\bigg)\left(2H_{0}[u](M-m)\right)^{-\frac{1}{2}}\\
\end{split}
\end{equation*}
\begin{equation*}
\begin{split}
&\quad-12\times\bigg(\frac{2a_{1}}{15}+\frac{4a_{2}}{35}M+\frac{16a_{2}}{105}m+\frac{2a_{3}}{21}M^{2}
+\frac{16a_{3}}{105}mM\\
&\quad\quad\quad\quad\quad+\frac{16a_{3}}{105}m^{2}\bigg)\left(2H_{0}[u](M-m)\right)^{\frac{1}{2}}\\
&\quad+\frac{2}{H_{0}[u]}\left(\frac{16a_{2}}{105}+\frac{16a_{3}}{105}M+\frac{32a_{3}}{105}m\right)(2H_{0}(M-m))^{\frac{3}{2}}.\\
\end{split}
\end{equation*}
In order to compute the derivatives of $E_{\varphi_{c}}$ at $(M_{\varphi_{c}},m_{\varphi_{c}})$, we take $E_{u}=E_{\varphi_{c}}$, $M=M_{\varphi_{c}}$
and $m=m_{\varphi_{_{c}}}$ in the above expressions for the partial derivatives of $E_{u}$
and use Lemma \ref{lem3.1}.

Next, we show that $(M_{\varphi_{c}},m_{\varphi_{c}})$ is an isolated local maximum of $E_{\varphi_{c}}$.
From the assumptions of Lemma \ref{lem3.4}, we deduce that the Hessian matrix of $E_{\varphi_{c}}$ at
the critical point $(M_{\varphi_{c}},m_{\varphi_{c}})$ is diagonal and negative definite,
which establishes its status as an isolated local maximum.
Hence this completes the proof of Lemma \ref{lem3.4}.

\end{proof}

\begin{lemma}\label{lem3.5} \cite{cll}
If $f\in H^{1}(\mathbb{S}^{1})$, then it satisfies
\begin{equation}\label{cai}
\max\limits_{x\in\mathbb{S}^{1}}|f(x)|\leq\sqrt{\frac{13}{12}}\|f\|_{\mu},
\end{equation}
where the $\mu$-norm is defined in \eqref{caa}. Moreover, $\sqrt{\frac{13}{12}}$ is the best constant and
equality holds in \eqref{cai} if and only if $f=c\varphi_{c}(\cdot-\xi+\frac{1}{2})$ for some $c$,
$\xi\in\mathbb{R}$, that is, if and only if $f$ has the shape of a peakon.

\end{lemma}

It is noticed that the equality in \eqref{cai} also holds for the $\mu$-CH, generalized $\mu$-CH and modified $\mu$-CH
peakons \cite{cll,lqz,qzll}. We recall the following lemma which shows that the $\mu$-norm is equivalent
to the $H^{1}(\mathbb{S}^{1})$-norm.

\begin{lemma}\label{lem3.6} \cite{cll}
Every $u\in C([0,T),H^{1}(\mathbb{S}^{1}))$ holds
\begin{equation*}
\|u\|_{\mu}^{2}\leq\|u\|_{H^{1}(\mathbb{S}^{1})}^{2}\leq 3\|u\|_{\mu}^{2}.
\end{equation*}
\end{lemma}

\begin{lemma}\label{lem3.7} \cite{cll,ce}
If $u\in C([0,T),H^{1}(\mathbb{S}^{1}))$, then
\begin{equation*}
M_{u(t)}=\max\limits_{x\in\mathbb{S}^{1}}u(t,x),\quad\quad m_{u(t)}=\min\limits_{x\in\mathbb{S}^{1}}u(t,x)
\end{equation*}
are continuous functions of $t\in[0,T)$.

\end{lemma}

\begin{lemma}\label{lem3.8}
Let $a>0$ and $a_{i}$, $i=1,2,3$ satisfy the assumptions of Lemma \ref{lem3.4}.
Let $u\in C([0,T),H^{1}(\mathbb{S}^{1}))$ be a solution of \eqref{mma}. Given a small neighborhood
$U$ of $(M_{\varphi_{c}}, m_{\varphi_{c}})\in \mathbb{R}^{2}$, there exists $\delta>0$ such that
\begin{equation}\label{caj}
(M_{u(t)},m_{u(t)})\in U\quad\quad \text{for}\quad t\in[0,T)\quad\text{if}\quad
\|u(\cdot,0)-\varphi_{c}\|_{H^{1}(\mathbb{S}^{1})}<\delta.
\end{equation}

\end{lemma}
\begin{proof}
Using the fact that the function $E_{u(t)}$ depends on $u$ only through three conservation laws $H_{0}[u]$,
$H_{1}[u]$ and $H_{2}[u]$. Therefore, $E_{u(t)}=E_{u}$ is independent of time $t$. Assume that
\begin{equation*}
H_{i}[u]=H_{i}[\varphi_{c}]+\varepsilon_{i},\quad\quad i=0,1,2.
\end{equation*}
A direct calculation gives rise to
\begin{equation*}
\begin{split}
&E_{u}(M,m)\\
&=\left(a_{1}M+a_{2}M^{2}+a_{3}M^{3}\right)\bigg[H_{1}[\varphi_{c}]+\varepsilon_{1}
+\frac{1}{2}\left(H_{0}[\varphi_{c}]+\varepsilon_{0}\right)^{2}
-m\left(H_{0}[\varphi_{c}]+\varepsilon_{0}\right)\\
&\quad\quad\quad\quad\quad\quad\quad\quad\quad\quad\quad\quad-\frac{4\sqrt{2}}{3}\left(H_{0}[\varphi_{c}]
+\varepsilon_{0}\right)^{\frac{1}{2}}(M-m)^{\frac{3}{2}}\bigg]+a_{1}m\left(H_{0}[\varphi_{c}]+\varepsilon_{0}\right)^{2}\\
&\quad+\left(\left(H_{0}[\varphi_{c}]+\varepsilon_{0}\right)m-\frac{23}{24}
\left(H_{0}[\varphi_{c}]+\varepsilon_{0}\right)^{2}\right)\left(\frac{721a_{2}}{5\times12^{2}}+\frac{60734a_{3}}{5\times7\times12^{3}}\right)\\
&\quad+4\sqrt{2}\bigg(\frac{2a_{1}}{15}m+\frac{a_{1}}{5}M+\frac{a_{2}}{7}M^{2}
+\frac{4a_{2}}{35}mM+\frac{8a_{2}}{105}m^{2}+\frac{a_{3}}{9}M^{3}+\frac{2a_{3}}{21}mM^{2}\\
&\quad\quad\quad\quad\quad+\frac{8a_{3}}{105}m^{2}M+\frac{16a_{3}}{315}m^{3}\bigg)
\left(H_{0}[\varphi_{c}]+\varepsilon_{0}\right)^{\frac{1}{2}}(M-m)^{\frac{3}{2}}-H_{2}[\varphi_{c}]-\varepsilon_{2}\\
&=E_{\varphi_{c}}(M,m)+\left(a_{1}M+a_{2}M^{2}+a_{3}M^{3}\right)\varepsilon_{1}-\varepsilon_{2}\\
&\quad+\Bigg[\left(a_{1}M+a_{2}M^{2}+a_{3}M^{3}\right)
\left(\frac{1}{2}+\frac{\sqrt{2}}{6}(M-m)^{\frac{3}{2}}H_{0}^{-\frac{3}{2}}[\varphi_{c}]\right)+a_{1}m\\
&\quad\quad\quad-\frac{23}{24}\left(\frac{721a_{2}}{5\times12^{2}}+\frac{60734a_{3}}{5\times7\times12^{3}}\right)
-\frac{\sqrt{2}}{2}(M-m)^{\frac{3}{2}}H_{0}^{-\frac{3}{2}}[\varphi_{c}]\\
&\quad\quad\quad\times\bigg(\frac{2a_{1}}{15}m+\frac{a_{1}}{5}M+\frac{a_{2}}{7}M^{2}
+\frac{4a_{2}}{35}mM+\frac{8a_{2}}{105}m^{2}+\frac{a_{3}}{9}M^{3}\\
&\quad\quad\quad\quad\quad+\frac{2a_{3}}{21}mM^{2}
+\frac{8a_{3}}{105}m^{2}M+\frac{16a_{3}}{315}m^{3}\bigg)\Bigg]\varepsilon_{0}^{2}\\
&\quad+\Bigg[\left(a_{1}M+a_{2}M^{2}+a_{3}M^{3}\right)\left(H_{0}[\varphi_{c}]-m
-\frac{2\sqrt{2}}{3}(M-m)^{\frac{3}{2}}H_{0}^{-\frac{1}{2}}[\varphi_{c}]\right)\\
&\quad\quad\quad+2a_{1}mH_{0}[\varphi_{c}]+\left(m-\frac{23}{12}H_{0}[\varphi_{c}]\right)
\left(\frac{721a_{2}}{5\times12^{2}}+\frac{60734a_{3}}{5\times7\times12^{3}}\right)\\
&\quad\quad\quad+2\sqrt{2}(M-m)^{\frac{3}{2}}H_{0}^{-\frac{1}{2}}[\varphi_{c}]\bigg(\frac{2a_{1}}{15}m+\frac{a_{1}}{5}M
+\frac{a_{2}}{7}M^{2}+\frac{4a_{2}}{35}mM+\frac{8a_{2}}{105}m^{2}\\
&\quad\quad\quad+\frac{a_{3}}{9}M^{3}+\frac{2a_{3}}{21}mM^{2}
+\frac{8a_{3}}{105}m^{2}M+\frac{16a_{3}}{315}m^{3}\bigg)\Bigg]\varepsilon_{0}
+o\left(\varepsilon_{0}^{2}\right),\\
\end{split}
\end{equation*}
where $o(\varepsilon_{0}^{2})$ represents the higher-order infinitesimal of $\varepsilon_{0}^{2}$. So $E_{u}$ is a small perturbation
of $E_{\varphi_{c}}$. The effect of the perturbation near the point $(M_{\varphi_{c}}, m_{\varphi_{c}})$
can be made arbitrarily small by choosing $\varepsilon_{i}$ $(i=0,1,2)$ small. In addition, we deduce from Lemma
\ref{lem3.4} that $E_{\varphi_{c}}(M_{\varphi_{c}}, m_{\varphi_{c}})=0$ and $E_{\varphi_{c}}$
has a critical point with negative definite second derivative at $(M_{\varphi_{c}}, m_{\varphi_{c}})$.
Applying the continuity of the second derivative, there is a neighborhood around $(M_{\varphi_{c}}, m_{\varphi_{c}})$,
where $E_{\varphi_{c}}$ is concave with curvature bounded away from zero. Hence, after a small perturbation,
the set $E_{\varphi_{c}}\geq0$ near $(M_{\varphi_{c}}, m_{\varphi_{c}})$ will be contained in a
neighborhood of $(M_{\varphi_{c}}, m_{\varphi_{c}})$.

Let $U$ be the neighborhood given as the above statement of the lemma. Shrinking $U$ if necessary, we obtain the existence
of $\delta'>0$ (depending on $U$) such that for $u\in C([0,T),H^{s}(\mathbb{S}^{1}))(s>3/2)$ with

\begin{equation}\label{cak}
|H_{i}[u]-H_{i}[\varphi_{c}]|<\delta',\quad\quad i=0,1,2,
\end{equation}
it holds that the set where $E_{u(t)}\geq 0$ is contained in $U$ for each $t\in[0,T)$.
It follows from Lemma \ref{lem3.3} and Lemma \ref{lem3.7} that $M_{u(t)}$ and $m_{u(t)}$ are continuous functions
of $t\in[0,T)$ and $E_{u(t)}(M_{u(t)},m_{u(t)})\geq0$ for $t\in[0,T)$. Then for
$u$ satisfying \eqref{cak}, one gets
\begin{equation*}
(M_{u(t)},m_{u(t)})\in U, \quad \text{for}\quad t\in[0,T)\quad \text{if}\quad (M_{u(0)},m_{u(0)})\in U.
\end{equation*}
However, the continuity of the conserved functionals $H_{i}:H^{1}(\mathbb{S}^{1})\rightarrow\mathbb{R}$,
$i=0,1,2$, verifies that there is $\delta>0$ such that \eqref{cak} holds for all $u$ with
\begin{equation*}
\|u(\cdot,0)-\varphi_{c}\|_{H^{1}(\mathbb{S}^{1})}<\delta.
\end{equation*}
Moreover, using the inequality \eqref{cai}, taking a smaller $\delta$ if necessary, we may
also assume that $(M_{u(0)},m_{u(0)})\in U$ if $\|u(\cdot,0)-\varphi_{c}\|_{H^{1}(\mathbb{S}^{1})}<\delta$.
This completes the proof of Lemma \ref{lem3.8}.
\end{proof}

We are now in a position to complete the proof of Theorem \ref{t3.1}.

\begin{proof} [{Proof of Theorem \ref{t3.1}}.]
Let $u\in C([0,T),H^{s})(s>3/2)$ be a solution of \eqref{mma} and suppose $\varepsilon>0$
is given. We choose a small neighborhood $U$ of $(M_{\varphi_{c}},m_{\varphi_{c}})$ such that
$|M-M_{\varphi_{c}}|<\frac{\varepsilon^{2}}{12a}$ if $(M,m)\in U$. Let us pick a $\delta>0$
as in Lemma \ref{lem3.8} so that \eqref{caj} holds. Taking a smaller $\delta$ if necessary,
we may assume that
\begin{equation*}
|H_{1}[u]-H_{1}[\varphi_{c}]|<\frac{\varepsilon^{2}}{12}\quad\text{if}\quad
\|u(\cdot,0)-\varphi_{c}\|_{H^{1}(\mathbb{S}^{1})}<\delta.
\end{equation*}
From Lemma \ref{lem3.2} and Lemma \ref{lem3.6}, it follows that
\begin{equation*}
\begin{split}
\|u(\cdot,t)-\varphi_{c}(\cdot-\xi(t))\|_{H^{1}(\mathbb{S}^{1})}^{2}
&\leq 3\|u(\cdot,t)-\varphi_{c}(\cdot-\xi(t))\|_{\mu}^{2}\\
&\leq6\left(H_{1}[u]-H_{1}[\varphi_{c}]\right)+6a|M_{\varphi_{c}}-M_{u(t)}|<\varepsilon^{2},\,\, t\in[0,T),\\
\end{split}
\end{equation*}
where $\xi(t)\in\mathbb{R}$ is any point and $u(\xi(t)+\frac{1}{2},t)=M_{u(t)}$.
This thus completes the proof of Theorem \ref{t3.1}.

\end{proof}

\vskip 0.5cm

\noindent {\bf Acknowledgments.} The work of Chong is supported by the National NSF of China Grant-11631007.
The work of Fu is supported by the National NSF of China Grants-11471259 and 11631007
and the National Science Basic Research Program of Shaanxi Province (Program No. 2019JM-007 and 2020JC-37).

\end{document}